\documentclass{article}
\usepackage[utf8]{inputenc}
\usepackage[]{amsmath,amssymb,amsthm,float}

\usepackage{graphicx}
\usepackage[]{algorithm2e}

\theoremstyle{plain}
\newtheorem{theorem}{Theorem}
\newtheorem{corollary}[theorem]{Corollary}
\newtheorem{lemma}[theorem]{Lemma}

\newcommand{\zp}{\mathbb{Z}/p\mathbb{Z}}

\newcommand{\OO}{\mathcal{O}}

\newcommand{\id}{\mathrm{Id}}

\title{A fast coset-translation algorithm for computing the cycle structure of Comer relation algebras over $\mathbb{Z}/p\mathbb{Z}$}
\author{Jeremy F. Alm\\Department of Mathematics\\Lamar University \\ Beaumont, TX 77710\\
\texttt{alm.academic@gmail.com}\\
\and Andrew Ylvisaker\\Austin, MN\\
\texttt{aylvisaker@gmail.com} }
\date{\today}

\begin{document}

\maketitle

\begin{abstract}
Proper relation algebras can be constructed using $\mathbb{Z}/p\mathbb{Z}$ as a base set using a method due to Comer.  The cycle structure of such an algebra must, in general, be determined \emph{a posteriori}, normally with the aid of a computer.  In this paper, we give an improved algorithm for checking the cycle structure that reduces the time complexity from $\mathcal{O}(p^2)$ to $\mathcal{O}(p)$.
\end{abstract}

\section{Introduction}

Comer \cite{comer83} introduced a technique for constructing finite integral proper relation algebras using $\zp$ as a base set for $p$ prime.  Set $p=nk+1$.  Then there is a multiplicative subgroup $H < (\zp)^\times$ of order $k$ and index $n$, and the subgroup $H$ can be used to construct a proper relation algebra of order $2^{n+1}$.  Specifically, fix $n\in \mathbb{Z}^+$, and let $X_0=H$ be the unique multiplicative subgroup of $(\zp)^\times$ of index $n$.  Let $X_1, \ldots X_{n-1}$ be its cosets; in particular, let $X_i = g^i\cdot X_0 = \{ g^{an+i}  : a \in \mathbb{Z}^+ \}$, where $g$ is a primitive root modulo $p$, i.e., $g$ is a generator of $(\zp)^\times$.  Then define relations 
\[
    R_i = \{(x,y)\in\zp\times\zp : x-y \in X_i\}. 
\]
 The $R_i$'s, along with $\id=\{(x,x):x\in\zp\}$, partition the set $\zp\times\zp$.  The $R_i$'s will be symmetric if $k$ is even and asymmetric otherwise.  For more information on relation algebras, see  \cite{Madd} or \cite{HH}. 
 
 Given a prime $p$ and integer $n$ a divisor of $p-1$, it is difficult  in general to say much about the cycle structure of the algebra generated by the $R_i$'s (although it is shown in \cite{Alm} that if $p>n^4+5$ then the cycles $(R_i,R_i,R_i)$ are mandatory); one must use a computational approach. The cycle $(R_i, R_j, R_k)$ is forbidden just in case $(X_i + X_j) \cap X_k = \varnothing$. Naively, one computes all the sumsets  $X_i + X_j $, though because of rotational symmetry, one may assume $i=0$:
 
 \begin{lemma}\label{lem:1}
 Let  $n \in \mathbb{Z}^{+}$ and let
 $p = nk + 1$ be a prime number and $g$ a primitive root modulo $p$. 
 
 For $i \in \left\{0,1,\ldots,n-1\right\}$, define

\[ X_{i} = \left\{g^{i},g^{n + i},g^{2n + i},\ldots,g^{(k-1)n + i}\right\}.\]

Then  $(X_0 + X_j) \cap X_k = \emptyset$ if and only if  $(X_i + X_{i+j}) \cap X_{i+k} = \emptyset$.
 \end{lemma}
 The lemma is trivial to prove: just multiply through by $g^i$!

  In 1983, Comer used his technique to construct (representations of) symmetric algebras with $n$ diversity atoms forbidding exactly the 1-cycles (i.e., the ``monochrome triangles'') for $1\leq n \leq 5$. Comer did the computations by hand, and his work was later extended via computer in \cite{Mad11} ($n \leq 7$), \cite{AlmManske2015} ($n \leq 400, \ n \neq 8,13 $), and \cite{Alm} ($401 \leq n \leq 2000$).  This last significant advance was made possible by a much less general version of the improved algorithm presented here.  Another variation was used in \cite{AlmAndrewsManske} to construct representations of algebras in which all the diversity atoms are flexible, and to give the first known cyclic group representation of relation algebra $32_{65}$.  Finally, an ad hoc variant was used in \cite{AlmMad} to give the first known finite representation of relation algebra $59_{65}$.

\section{Symmetric Comer algebras}

The following na\"ive algorithm, used in \cite{AlmManske2015} to find Ramsey algebras for $n \leq 400, \ n \neq 8,13 $, computes all the sumsets $X_0 + X_i$.

\begin{algorithm}[H]\label{alg1}
 \KwData{A prime $p$, a divisor $n$ of $(p-1)/2$ such that $(p-1)/n$ is even, a primitive root $g$ modulo $p$}
 \KwResult{a list of mandatory and forbidden cycles of the form $(0,x,y)$}
 
 Compute $X_0 = \{g^{an} \pmod{p} : 0\leq a < (p-1)/n \}$;

 \For{$i\leftarrow 0$ \KwTo $n-1$}{
    $X_i = \{g^{an+i} \pmod{p} : 0\leq a < (p-1)/n \}$;
    
    Compute $X_0 + X_i \pmod{p}$;
    
    \For{$j\leftarrow i$ \KwTo $n-1$}{
        Compute $X_j = \{g^{an+j} \pmod{p} : 0\leq a < (p-1)/n \}$;
        
        \uIf{$(X_0 + X_i)\supseteq X_j$}{
            add $(0,i,j)$ to list of mandatory cycles
            }
            
            \Else{
            add $(0,i,j)$ to list of forbidden cycles
            }
    
    }
 }
 
\caption{Na\"ive algorithm for symmetric Comer algebras}
\end{algorithm}

The following lemma is very easy to prove and was apparently known Comer.  The algorithmic speed-up it provides, however, was not previously noticed.

\begin{lemma} \label{lem2}
Let  $n \in \mathbb{Z}^{+}$ and let
 $p = nk + 1$ be a prime number, $k$ even, and $g$ a primitive root modulo $p$. 
 For $i \in \left\{0,1,\ldots,n-1\right\}$, define

\[ X_{i} = \left\{g^{i},g^{n + i},g^{2n + i},\ldots,g^{(k-1)n + i}\right\}.\]

Then if $(X_0 + X_i) \cap X_j \neq \emptyset$, then $(X_0 + X_i) \supseteq X_j$. 
\end{lemma}

In particular, Lemma \ref{lem2} implies that Comer's construction always yields a relation algebra. 

\begin{corollary}\label{cor}
A diversity cycle $(X_0,X_i,X_j)$ is forbidden if and only if $(g^j - X_0)\cap X_i = \emptyset$. 
\end{corollary}

Corollary \ref{cor} affords us the following faster algorithm for computing the cycle structure of Comer relation algebras.

\begin{algorithm}[H] \label{alg2}
 \KwData{A prime $p$, a divisor $n$ of $(p-1)/2$, a primitive root $g$ modulo $p$}
 \KwResult{a list of mandatory and forbidden cycles of the form $(0,x,y)$}
 Compute $X_0 = \{g^{an} \pmod{p} : 0\leq a < (p-1)/n \}$;
 
 Compute $g^j - X_0 \pmod{p}$ for each $0\leq j < n$;
 
 \For{$i\leftarrow 0$ \KwTo $n-1$}{
    $X_i = \{g^{an+i} \pmod{p} : 0\leq a < (p-1)/n \}$
    
    \For{$j\leftarrow i$ \KwTo $n-1$}{
        \eIf{$(g^j - X_0)\cap X_i \neq \emptyset$}{
            add $(0,i,j)$ to list of mandatory cycles
            }{
            add $(0,i,j)$ to list of forbidden cycles
            }
    
    }
 }
 
  \caption{Fast algorithm for symmetric Comer algebras}
\end{algorithm}

 For example, let $p=113$, $n=7$.  Then $k=16$. Since $k$ is even, we get a symmetric algebra, i.e. $X_i = -X_i$.  The forbidden cycles that Algorithm \ref{alg2} spits out are $(0,0,0)$, $(0,0,4)$, and $(0,3,3)$.  Note that $(0,0,4)$ and $(0,3,3)$ are equivalent by Lemma \ref{lem:1}.  (Here we are using $g=3$, the smallest primitive root modulo $p$.)

 For another example, let $p=71$, $n=10$.  Then $k=7$ is odd, so we get an asymmetric algebra with five pairs $X_i, X_{i+5}$ such that $X_i = -X_{i+5}$.  The forbidden cycles are as follows:
 \[
 \begin{tabular}{ccc}
     (0, 0, 0) & (0, 0, 3) & (0, 0, 4) \\
(0, 0, 5) & (0, 0, 8) & (0, 0, 9) \\
(0, 1, 2) & (0, 1, 4) & (0, 1, 6) \\
(0, 1, 7) & (0, 2, 3) & (0, 2, 6) \\
(0, 2, 7) & (0, 2, 9) & (0, 3, 3) \\
(0, 3, 5) & (0, 3, 9) & (0, 4, 5) \\
(0, 4, 7) & (0, 4, 8) & (0, 5, 5) \\
(0, 5, 7) & (0, 6, 8) & (0, 7, 8) \\
(0, 7, 9) & (0, 8, 8) &
 \end{tabular}
 \]
 
 Here we are using $g=7$, the smallest primitive root modulo $p$. Note that the choice of primitive root does affect how the various cosets are indexed -- in particular, one always has $g\in X_1$ --  but it does not affect which relation algebra one gets.


\begin{theorem}
Algorithm \ref{alg1} runs in $\OO(p^2)$ time while Algorithm \ref{alg2} runs in $\OO(p)$ time, for $n$ fixed.
\end{theorem}

\begin{proof}
In Algorithm \ref{alg1}, each pass through the inner loop requires $\OO(k)$ comparisons, and each pass through the outer loop requires $\OO(k^2)$ additions.  So overall, $\OO(nk^2)$ additions and $\OO(n^2k)$ comparisons are required, for an overall runtime of $\OO(p^2)$ for $n$ fixed.

In Algorithm \ref{alg2}, each pass through the inner loop requires $\OO(k)$ additions and $\OO(k)$ comparisons.  So overall, $\OO(n^2k)$ additions and $\OO(n^2k)$ comparisons are required, for an overall runtime of $\OO(p)$ for $n$ fixed.
\end{proof}

Algorithm \ref{alg2} is  much faster in practice.  Both algorithms were implemented by the first author in Python 3.4. Timing data were collected for primes $p\equiv 1 \pmod{23}$ under 15,000 on an Intel Core i5-7500T $@$ 2.7GHz. See Figure \ref{fig:my_label}. The quadratic nature of Algorithm \ref{alg1} is evident.

\begin{figure}[hbt]
    \centering
    \includegraphics[width=4.5in]{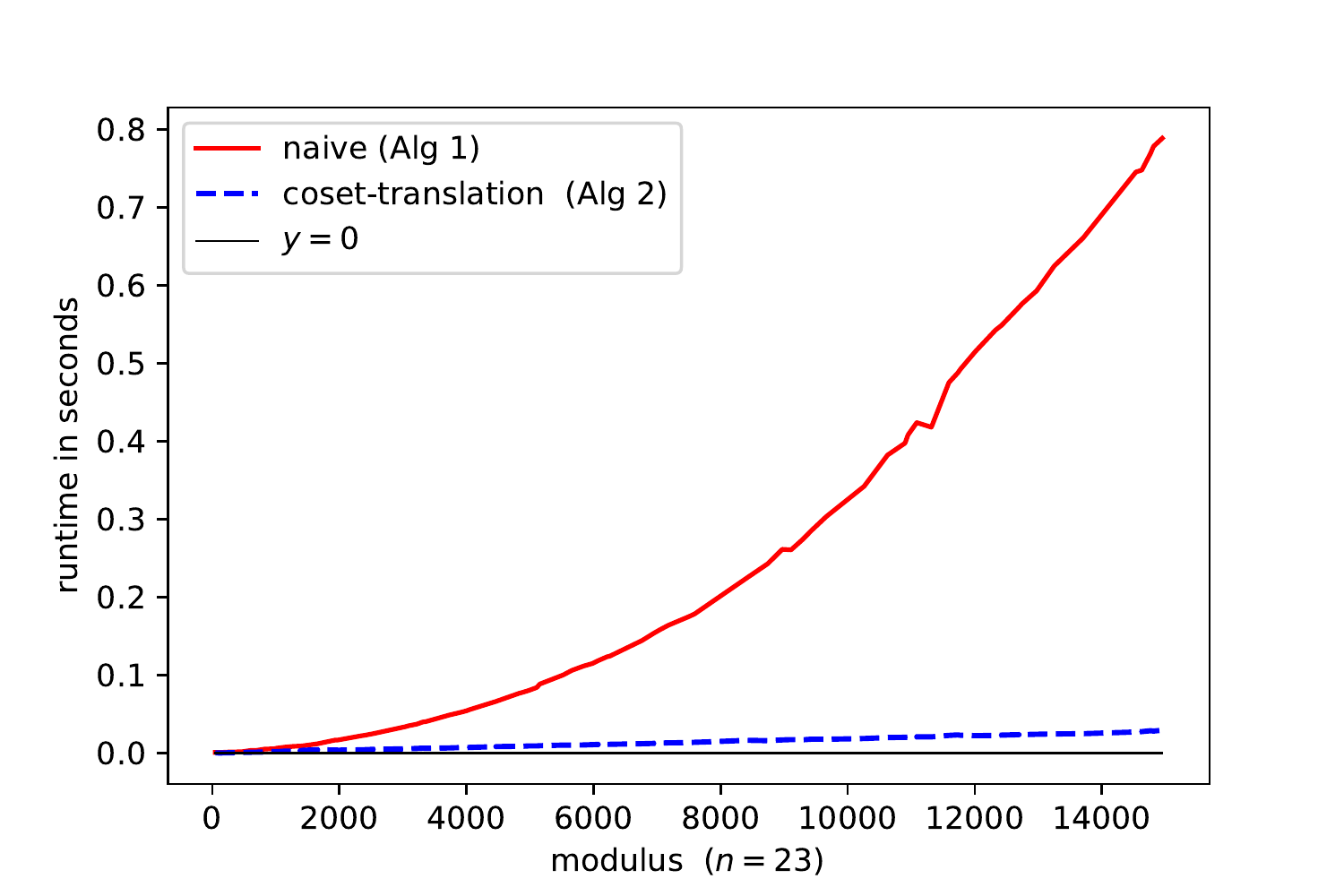}
    \caption{Run-time comparison for Algorithms 1 \& 2}
    \label{fig:my_label}
\end{figure}

\section{Asymmetric Comer algebras}

For the case of asymmetric algebras, we need to take a little more care in our enumeration over indices $i,j$ in checking whether $(X_0 + X_i) \supseteq X_j$. Let $n$ be even, where $n=2m$.   Since $-X_i = X_{i+m}$, where all indices are computed mod $n$, we have the following equivalence:
\begin{equation}\label{equiv}
    (X_0 + X_i) \supseteq X_j \longleftrightarrow (X_0 + X_{j+m}) \supseteq X_{i+m}.
\end{equation}
Thus for every triple $(0,i,j)$ of indices, there is an equivalent triple $(0,j+m,i+m)$ that would be redundant to check. So consider the involution $(0,i,j)\mapsto (0,j+m,i+m)$ on triples of indices.  The fixed points of this involution are of the form $(0,i,i+m)$.  (Of course, we continue to compute indices mod $n$.) Consider an $n\times n$ matrix $A$ where the entry $A_{ij} = (0,i,j+m) $. For example, see the matrix below, where $n=6$:

\[
  \begin{bmatrix}
    (0,0,3) & (0,0,4) & (0,0,5) & (0,0,0) & (0,0,1) & (0,0,2)  \\
    (0,1,3) & (0,1,4) & (0,1,5) & (0,1,0) & (0,1,1) & (0,1,2)  \\
    (0,2,3) & (0,2,4) & (0,2,5) & (0,2,0) & (0,2,1) & (0,2,2)  \\
    (0,3,3) & (0,3,4) & (0,3,5) & (0,3,0) & (0,3,1) & (0,3,2)  \\
    (0,4,3) & (0,4,4) & (0,4,5) & (0,4,0) & (0,4,1) & (0,4,2)  \\
    (0,5,3) & (0,5,4) & (0,5,5) & (0,5,0) & (0,5,1) & (0,5,2)  \\
  \end{bmatrix}
\]
Then the fixed points of the involution are on the diagonal, and $A_{ij}$ is equivalent to $A_{ji}$ by the equivalence \eqref{equiv}. Thus it suffices to enumerate over the ``upper triangle'' of the matrix.

In Algorithm \ref{alg3} below, the enumeration is done according to the discussion in the previous paragraph. 

\begin{algorithm}[H] \label{alg3}
 \KwData{A prime $p$, a divisor $n$ of $(p-1)$ such that $(p-1)/n$ is odd, a primitive root $g$ modulo $p$}
 \KwResult{a list of mandatory and forbidden cycles of the form $(0,x,y)$}
 Let $m = n/2$.
 Compute $X_0 = \{g^{an} \pmod{p} : 0\leq a < (p-1)/n \}$;
 
 Compute $g^j - X_0 \pmod{p}$ for each $0\leq j < n$;
 
 \For{$i\leftarrow 0$ \KwTo $n-1$}{
    $X_i = \{g^{an+i} \pmod{p} : 0\leq a < (p-1)/n \}$
    
    \For{$j\leftarrow i+m \pmod{n}$ \KwTo $n+m-1 \pmod{n}$}{
        \eIf{$(g^j - X_0)\cap X_i \neq \emptyset$}{
            add $(0,i,j)$ to list of mandatory cycles
            }{
            add $(0,i,j)$ to list of forbidden cycles
            }
    
    }
 }

 \caption{Fast algorithm for asymmetric Comer algebras}
\end{algorithm}

\section{Acknowledgements}
 We thank Wesley Calvert and Southern Illinois University for hosting the 2016 Langenhop Lecture and Mathematics Conference. The idea for this algorithmic improvement was obtained by the authors over a lunch break at the conference.



\end{document}